\documentclass[10pt]{amsart} \usepackage{xcolor,amssymb,amsmath}
\usepackage[utf8x]{inputenc} \newtheorem{tm}{Theorem}[section]
\newtheorem{con}{Conjecture} 
\newtheorem{prop}[tm]{Proposition} \newtheorem{lm}[tm]{Lemma}
\usepackage{xcolor,ulem}
\newcommand{\bP}{\ensuremath{\mathbb P}} \DeclareMathOperator{\HS}{HS}

\title[WLP for ACI's generated by uniform powers of general linear
forms] {A classification of the Weak Lefschetz property for almost
  complete intersections generated by uniform powers of general linear
  forms} \subjclass[2010]{Primary: 13E10, 14C20; Secondary: 13C40,
  13C13, 13D40} \keywords{powers of linear forms, general linear
  forms, almost complete intersections, Weak Lefschetz property,
  inverse systems, Hilbert series} \author{Mats Boij}
\email{boij@kth.se} \address{Department of Mathematics\\ KTH - Royal
  Institute of Technology\\ SE-100 44 Stockholm, Sweden}
\author{Samuel Lundqvist} \email{samuel@math.su.se}
\address{Department of Mathematics\\ Stockholm University\\
  SE-106 91 Stockholm, Sweden } \date{}

\begin{document}
\begin{abstract}
  We use Macaulay's inverse system to study the Hilbert series for
  almost complete intersections generated by uniform powers of general
  linear forms. This allows us to give a classification of the Weak
  Lefschetz property for these algebras, settling a conjecture by
  Migliore, Mir\'o-Roig, and Nagel.
\end{abstract}

\maketitle

\section{Introduction}
Let $\Bbbk$ be a field of characteristic zero and let
$\ell_1,\ell_2,\ldots,\ell_r \in \Bbbk[x_1,x_2,\ldots,x_n]$ be general
linear forms. For positive integers $d_1,d_2,\ldots,d_r$, consider the
quotient
$\Bbbk[x_1,x_2,\ldots,x_n]/ \langle
\ell_1^{d_1},\ell_2^{d_2},\ldots,\ell_r^{d_r} \rangle$. This algebra
ties together several areas of contemporary mathematics.

From the algebraic point of view, it is in a natural way linked to the
long-standing conjectures by Fr\"oberg \cite{F} and Iarrobino
\cite{iarrobino} on the Hilbert series of generic forms.

Powers of general linear forms are also tightly connected to the study
of fat points schemes via Macaulay's inverse system, as was noticed by
Emsalem and Iarrobino \cite{emsalem}. This bridge to geometry relates
the study of powers of general linear forms to the Alexander-Hirschowitz 
Theorem \cite{ah, chandler} and the Segre–Gimigliano–Harbourne–Hirschowitz (SGHH) 
Conjecture \cite{sghh}.

In this paper we consider the uniform almost complete intersection
case, that is, algebras of the form
$\Bbbk[x_1,x_2,\ldots,x_n]/(\ell_1^{d},\ell_2^{d},\ldots,\ell_{n+1}^{d}),$
from the perspective of the Weak Lefschetz property.

Recall that a graded algebra $A$ satisfies the Weak Lefschetz property
(WLP) if there exists a linear form $\ell$ such that the
multiplication map $\times \ell:A_i \to A_{i+1}$ has maximal rank for
all degrees $i$, while $A$ satisfies the Strong Lefschetz property
(SLP) if the multiplication map $\times \ell^j:A_i \to A_{i+j}$ has
maximal rank for all $i$ and all $j$.  For an introduction to the
Lefschetz properties, see e.g. \cite{book, atour}.

The WLP for the class of algebras that we consider holds for $n = 1$
and $n=2$, since all graded artinian quotients in one or two variables
have the SLP, the argument being trivial for the univariate case,
while the case of two variables, which requires characteristic zero,
is attributed to Harima, Migliore, Nagel, and Watanabe \cite{HMNW}.
For $n=3$, Schenck and Seceleanu \cite{edim3} showed that the WLP
holds for \emph{any} quotient by an artinian ideal generated by powers
of linear forms. Migliore, Mir\'o-Roig, and Nagel \cite{MMN} showed
that for even $n \geq 4$, the WLP fails for almost complete
intersections generated by uniform powers, except in the case
$(n,d) = (4,2)$. They also gave results in the odd uniform case and in
the mixed degree case, and provided a conjecture for the unproven part
of the odd uniform case.
 
To simplify the statement of the conjecture and the presentation in
this paper in general, we let $R_{n,m,d}$ denote the ring
$\Bbbk[x_1,x_2,\ldots,x_n]/\langle \ell_1^d,\ell_2^d,\ldots,\ell_{m}^d
\rangle$, where $\ell_1,\ell_2,\ldots,\ell_m$ are general linear forms
and where $\Bbbk$ is a field of characteristic zero.

\begin{con} \cite[Conjecture 6.6]{MMN} \label{con1} Let $n\geq9$ be an
  odd integer. Then
  $R_{n,n+1,d}=\Bbbk[x_1,x_2,\ldots,x_n]/\langle
  \ell_1^d,\ell_2^d,\ldots,\ell_{n+1}^d \rangle$ fails the WLP if and
  only if $d>1$. Furthermore, if $n=7$ then $R_{n,n+1,d}$ fails the
  WLP when $d=3$.
\end{con}

Independently, Harbourne, Schenck, and Seceleanu \cite{harbourne} gave
a more general but less precise conjecture.

\begin{con} \cite[Conjecture 1.5]{harbourne} \label{con2} Let
  $r+1\geq n \geq5$. Then the algebra
  $R_{n,r,d}=\Bbbk[x_1,x_2,\ldots,x_n]/\langle
  \ell_1^d,\ell_2^d,\ldots,\ell_{r}^d \rangle$ fails the WLP if
  $d \gg 0$.
\end{con}

Since then, the main focus have been on Conjecture
\ref{con1}. Mir\'o-Roig \cite{M} has shown the failure of the WLP when
$d=2$, Nagel and Trok \cite{NT} have shown the failure when both $n$
and $d$ are large enough, and also when $n \geq 9, d-2 \gg 0$ and
$d-2$ is divisible by $n$.  Ilardi and Vall\`es \cite{edim7} have
settled the case $(n,d) = (7,3)$, while Mir\'o-Roig and Tran \cite{MT}
have shown the failure in the cases $9 \leq n=2m+1 \leq 17, d \geq 4$,
and in the cases $d=2r, 1 \leq r \leq 8, 9 \leq n \leq 4r(r+2)-1.$

In this paper, we cover all the remaining cases, settling Conjecture
\ref{con1} and provide the following classification.
\begin{tm} \label{thm:main} Let $d,n\geq 1$. Then
  $R_{n,n+1,d}=\Bbbk[x_1,x_2,\ldots,x_n]/\langle
  \ell_1^d,\ell_2^d,\ldots,\ell_{n+1}^d \rangle$ fails the WLP except
  when $n\le 3$, $d=1$ or
  $(n,d) \in \{ (4,2), (5,2), (5,3), (7,2) \},$ and in these cases,
  the WLP holds.
\end{tm}

The paper is organized as follows. In Section \ref{sec:prel} we
introduce some notation that will be used subsequently. In Section
\ref{sec:degree} we use the theory for inverse systems to determine
the degree of the Hilbert series for $R_{n,n+2,d}$. In Section
\ref{sec:inflection} we give an upper bound for the degree of the
Hilbert series for $R_{n,n+2,d}$ under the assumption that
$R_{n+1,n+2,d}$ has the WLP. By comparing this upper bound with the
actual degree, we can draw the conclusion that $R_{n+1,n+2,d}$ fails
the WLP in all but a finite number of cases. The remaining cases are
then dealt with separately in Section \ref{sec:sporadic}.

\section{Preliminaries} \label{sec:prel}

We begin by giving some background on Hilbert series, Fr\"oberg's
conjecture, and inverse systems.

The Hilbert series for a standard graded algebra
$A = \bigoplus_{i \geq 0} A_i$ is the power series
$\sum \dim_{\Bbbk} A_i t^i$ and is denoted by $\HS(A,t)$. The Hilbert
function of $A$ is the function $i \mapsto \dim_{\Bbbk} A_i$.

If $A$ is an artinian graded algebra and $f$ is a form in $A$ of
degree $d$ such that the map $\times f: A_i \to A_{i+d}$ has maximal
rank for all $i$, then it is an easy exercise to check that the
Hilbert series for $A/(f)$ is equal to
$\left[\HS(A,t) \cdot (1-t^d) \right]$.  Here the bracket notation
means that we truncate the series before the first non-positive term.

Fr\"oberg \cite{F} has conjectured that if $A$ is the polynomial ring
modulo an ideal generated by general forms, then the map induced by
multiplication by a general form of degree $d$ has maximal
rank. Normally, this conjecture is expressed equivalently as follows:
if $f_1,\ldots,f_r$ are general forms in $\Bbbk[x_1,x_2,\ldots,x_n]$
of degrees $d_1,d_2,\ldots,d_r$, then the Hilbert series for
$\Bbbk[x_1,x_2,\ldots,x_n]/ \langle f_1,f_2,\ldots,f_r \rangle$ equals
$\left[\prod (1-t^{d_i})/(1-t)^n \right]$.

In \cite{F} Fröberg also proves thatWe will in t
$\left[\prod (1-t^{d_i})/(1-t)^n\right]$ is a lower bound for possible
Hilbert series among forms of degrees $d_1,d_2,\ldots,d_r$ in the
lexicographic sense, so for a fixed signature
$(n,d_1,d_2,\ldots,d_r)$, the conjecture can be verified with an
example.

The conjecture is, except for a few cases, open for $r-1 > n \geq
4$. For some recent results, see \cite{N}. The case $r = n +1$ is due
to Stanley \cite{Stanley} and is of particular importance for this
paper.

Let $A$ be a monomial complete intersection; 
$A =  \Bbbk[x_1,x_2, \ldots,x_n]/(x_1^{d_1}, x_2^{d_2},\ldots,x_n^{d_n})$. 
Stanley showed that the multiplication
map $\times (x_1+x_2+\cdots+x_n)^d: A_i \to A_{i+d}$ has full rank for
every $i$ and $d$, not only settling the $n+1$--case of the Fr\"oberg
conjecture, but also opening up the area of the Lefschetz properties
for graded algebras. If we perform a linear change of coordinates,
Stanley's result is equivalent to the fact that complete intersections
generated by powers of general linear forms have the SLP.

When restricted to the equigenerated case $d=d_1= \ldots = d_{n+1}$,
this implies that
$$\HS(R_{n,n+1,d}, t) = \left[\frac{(1-t^d)^{n+1}}{(1-t)^n} \right].$$
 
Suppose now that $R_{n,n+1,d}$ satisfies the WLP. Then the map induced
by multiplication by a general linear form $\ell$ has maximal rank in
every degree, so the Hilbert series for $R_{n,n+1,d}/(\ell)$ equals
$$
\left[(1-t) \left[\frac{(1-t^d)^{n+1}}{(1-t)^n} \right] \right] =
\left[ (1-t) \frac{(1-t^d)^{n+1}}{(1-t)^n} \right] =
\left[\frac{(1-t^d)^{n+1}}{(1-t)^{n-1}} \right],
$$ 
where the first equality follows from \cite[Lemma 4]{F}.

Since $R_{n,n+1,d}/(\ell)$ is isomorphic to $R_{n-1,n+1,d}$, this
gives that $R_{n,n+1,d}$ has the WLP if and only if the Hilbert series
of $R_{n-1,n+1,d}$ is the one expected by Fr\"oberg's conjecture, that
is,
\begin{equation} \label{eq:WLPred} R_{n,n+1,d} \text{ has the WLP if
    and only if } \HS(R_{n-1,n+1,d}, t) =
  \left[\frac{(1-t^d)^{n+1}}{(1-t)^{n-1}} \right].
\end{equation}

For an ideal $I$ in the polynomial ring $\Bbbk[x_1,x_2,\dots,x_n]$ we
consider the dual polynomial ring $\Bbbk[X_1,X_2,\dots,X_n]$ where
$x_i$ acts like $\partial /\partial X_i$, for $i=1,2,\dots,n$, and the
inverse system of $I$, denoted by $I^{-1}$, is the submodule
annihilated by $I$ under this action. We use the notation $f\circ F$
for the action of the form $f\in \Bbbk[x_1,x_2,\dots,x_n]$ on the form
$F\in \Bbbk[X_1,X_2,\dots,X_n]$. By duality,
\begin{equation} \label{eq:inverse} \dim_\Bbbk [I^{-1}]_d =
  \dim_\Bbbk[\Bbbk[x_1,x_2,\dots,x_n]/I]_d,
\end{equation}
and this will enable us to use the inverse system in order to obtain
lower bounds for the Hilbert series.

Finally, the action of the $d$'th power of a general linear form
$\ell$ on a form $F$ in $\Bbbk[X_1,X_2,\dots,X_n]$ will be of
particular importance to us, and we therefore recall the general
Leibniz rule
$$\ell^d \circ F^n = \sum_{d_1+d_2+\cdots+d_n=d} \frac{d!}{d_1! d_2! \cdots d_n!} (\ell^{d_1} \circ F) (\ell^{d_2} \circ F) \cdots (\ell^{d_n} \circ F).$$

\section{The degree of the Hilbert series for $R_{n,n+2,d}$
}\label{sec:degree}

By definition, the degree of the Hilbert series for an artinian graded
algebra $A$ is equal to max $\{j \,|\, A_j \neq 0\}$. Since $A$ is
artinian and non-zero, this number also agrees with the
Castelnuovo-Mumford regularity of $A$, see \cite{E}.

Let $$s(n,d) = \begin{cases}
  \frac{(n+1)(d-1)}{2} & \text{if $n$ is odd,}\\
  \left \lfloor \frac{n(n+2)(d-1)}{2(n+1)} \right \rfloor & \text{if $n$ is even.}\\
\end{cases}
$$

We will show that $\deg (\HS(R_{n,n+2,d},t)) = s(n,d)$ for all
$n,d\geq 1$.

Sturmfels and Xu \cite{squares} have shown that
$\deg(\HS(R_{n,n+2,2,},t)) = s(n,2)$, and that the dimension of
$R_{n,n+2,2}$ in degree $s(n,2)$ is equal to $2^{\frac{n}{2}}$ if $n$
is even, and equal to $1$ if $n$ is odd.

Nagel and Trok \cite{NT} proved that
$\deg(\HS(R_{n,n+2,d},t)) \leq s(n,d)$. They also proved equality in
the case $n$ odd, in which the dimension of $R_{n,n+2,d}$ in degree
$s(n,d)$ is equal to $1$, and in the case $n$ even and $n+1$ divides
$d-1$ or $d\geq n^2+n+2$, in which the dimension of $R_{n,n+2,d}$ in
degree $s(n,d)$ is equal to a binomial coefficient.

By (\ref{eq:inverse}), we have
$$\dim_\Bbbk [ \langle \ell_1^d,\ell_2^d,\ldots,\ell_{n+2}^d \rangle^{-1}]_s  \neq 0 \Longrightarrow  \deg (\HS(R_{n,n+2,d},t)) \geq s,$$
so in order to show that $s(n,d)$ is a lower bound, it is enough to
show that the inverse system is non-zero in degree $s(n,d)$.

Although it is sufficient to show that $s(n,d)$ is a lower bound for
$\deg(\HS(R_{n,n+2,d},t))$ in the unproven part of the case $n$ even,
we show, for completeness, that $s(n,d)$ is a lower bound for all $n$.

We begin with an alternative proof of the case $n$ odd. The argument
is short and also gives the main idea behind the more involved proof
for the even case.

\begin{prop} \label{prop:nodd} Let $n\geq1$ be odd and let $d \geq
  1$. Then the value of the Hilbert function of $R_{n,n+2,d}$ is
  non-zero in degree $s(n,d)$.
\end{prop}
\begin{proof}
  When $d = 1$, we have $s(n,1) = 0$, and the value of the Hilbert
  function in degree $0$ is equal to $1$. The case $d=2$ follows from
  the result by Sturmfels and Xu, in particular, there is a form $F$
  of degree $s(n,d)$ such that $\ell_i^2 \circ F = 0$ for
  $i = 1, \ldots, n+2$. For the case $d > 2$, it follows from the
  pigeonhole principle in conjunction with the general Leibniz rule
  that $\ell_i^d \circ F^{d-1} = 0$ for $i = 1, \ldots, n+2$. Finally,
  the degree of $F^{d-1}$ equals
  $$(d-1) s(n,2) = (d-1) \frac{n-1}{2} = s(n,d).$$
\end{proof}

We now turn to the even case. Also here we are able to reduce the
argument to the result by Sturmfels and Xu in degree $2$.

\begin{lm} \label{lemma:dleq3} Let $n$ be even and let
  $1 \leq d \leq 3$. Then the value of the Hilbert function of
  $R_{n,n+2,d}$ is non-zero in degree $s(n,d)$.

\end{lm}
\begin{proof}
  The cases $d=1$ and $d= 2$ are dealt with similarly as in
  Proposition \ref{prop:nodd}.

  We now consider the case $d = 3$. Let $F$ be a form such that
  $\ell_i^2 \circ F = 0$ for all $i$. Since $F$ is in the inverse
  system of the ideal generated by $n+2$ squares of general linear
  forms, we can choose $F$ of degree $s(n,2)$.

  By the pigeonhole principle and the general Leibniz rule, we have
  $\ell_i^3 \circ F^2 = 0$. The degree of $F^2$ is $2s(n,2)$. Since
$$s(n,2)= \left \lfloor \frac{n(n+2)}{2(n+1)}  \right \rfloor = 
\left \lfloor \frac{n(n+1)}{2(n+1)} + \frac{n}{2(n+1)} \right \rfloor
= \left \lfloor \frac{n}{2} + \frac{n}{2(n+1)} \right \rfloor =
\frac{n}{2}
$$
and
$$s(n,3) 
= \left \lfloor \frac{n(n+2)}{n+1} \right \rfloor = \left \lfloor n +
  \frac{n+1}{n+2} \right \rfloor = n,
$$
we get that the degree of $F^2$ equals $s(n,3)$, which shows that the
value of the Hilbert function is non-zero in degree $s(n,3)$.

\end{proof}

\begin{lm} \label{lemma:dg3} Let $n$ be even and let
  $4 \leq d \leq n+1$. Then the value of the Hilbert function of
  $R_{n,n+2,d}$ is non-zero in degree $s(n,d)$.
\end{lm}

\begin{proof}
  Let $F_i$ be such that $\ell_i \circ F_i = 0$ and
  $\ell_j^2 \circ F_i = 0$ for all $j$.  Since $F_i$ is in the inverse
  system of the ideal generated by $1$ general linear form and $n+1$
  squares of general linear forms, we can choose $F_i$ to be of degree
  $s(n-1,2)$.

  Next, let $G$ be such that $\ell_i \circ G = 0$ for $i \geq d$, and
  $\ell_i^2 \circ G = 0$ for all $i$. Now $G$ is in the inverse system
  of the ideal generated by $n+2-d+1$ general linear forms and $d-1$
  squares of general linear forms, so we can choose $G$ of degree
  $s(n-(n+2-d+1),2) = s(d-3,2)$.

  It follows by the pigeonhole principle and the general Leibniz rule
  that $\ell_i^{d} \circ G F_1 \cdots F_{d-1} = 0$ for
  $i=1,\ldots,n+2$, and we are done if we can show that the degree of
  the form $G F_1 \cdots F_{d-1}$ is equal to $s(n,d)$, that is, that
  $s(d-3,2) + (d-1) s(n-1,2) = s(n,d)$.

  Suppose first that $d$ is odd and write $d-1 = 2c$.  We get
  \begin{align*}
    s(d-3,2)  &= \left \lfloor \frac{(d-1)(d-3)}{2(d-2)} \right \rfloor = 
                \left \lfloor \frac{c (2c-2)}{2c-1} \right \rfloor = 
                \left \lfloor c-\frac{c}{2c-1} \right \rfloor = c - 1,\\
    (d-1) \cdot s(n-1,2) &= c n,\\
    s(n,d) &= \left \lfloor \frac{ (n+2) n c}{n+1} \right \rfloor = 
             \left \lfloor n c + \frac{nc}{n+1} \right \rfloor = 
             nc + \left \lfloor c - \frac{c}{n+1} \right \rfloor = nc + c - 1,\\
  \end{align*}

  where we in the last step have used that $c < n+1$. This proves the
  case $d$ odd.

  Suppose now that $d$ is even. We get
  \begin{align*}
    s(d-3,2) &= \frac{d-2}{2} \\
    (d-1) \cdot s(n-1,2) &= \frac{n(d-1)}{2}
  \end{align*}
  and
  \begin{align*}
    s(n,d) &= \left \lfloor \frac{(n+2)n(d-1)}{2 (n+1)} \right \rfloor = 
             \left \lfloor \frac{n(d-1)}{2} + \frac{n(d-1)}{2(n+1} \right \rfloor =
             \frac{n(d-1)}{2} + \left \lfloor \frac{n(d-1)}{2(n+1)} \right \rfloor \\
           &= \frac{n(d-1)}{2} + \left \lfloor \frac{d-1}{2} - \frac{d-1}{2(n+1)} \right \rfloor \\
           &= \frac{n(d-1)}{2} + \left \lfloor \frac{d}{2} - \frac{1}{2} - \frac{d-1}{2(n+1} \right \rfloor =  
             \frac{n(d-1)}{2} + \frac{d}{2} -1,\\
  \end{align*}
  where we in the last step have used that $d-1 < n + 1$. This
  finishes the proof.
\end{proof}

\begin{tm} \label{thm:degree} The degree of the Hilbert series for
  $R_{n,n+2,d}$ equals $s(n,d)$.
\end{tm}
\begin{proof}
  The case $n$ odd was established by Nagel and Trok, so we only need
  to consider the case $n$ even.  Moreover, by \cite[Theorem 4.4]{NT},
  the degree of the Hilbert series of $R_{n,n+2,d}$ is less than or
  equal to $s(n,d)$, so it is sufficient to prove that the
  $R_{n,n+2,d}$ is non-zero in degree $s(n,d)$.

  Write $d = c + a(n+1)$, where $1 \leq c \leq n+1$. By Lemma
  \ref{lemma:dleq3} and Lemma \ref{lemma:dg3}, there is a form $F_c$
  of degree $s(n,c)$ such that $\ell_i^{c} \circ F_c = 0$ for
  $i = 1, \ldots, n+2$. Let $F = F_1 \cdots F_{n+2}$, where $F_i$ is
  such that $\ell_i \circ F_i = 0$ and $\ell_j^2 \circ F_i = 0$ for
  all $j$.  Then, by the pigeonhole principle and the general Leibniz
  rule, we have that $\ell^{(n+1) +1} \circ F = 0$, or more
  generalized, that $\ell^{a(n+1) +1} \circ F^a = 0$.

  It follows that $\ell_i^{c+a(n+1)} \circ F_c F^a = 0$. Thus we are
  done if we can show that the degree of $F_c F^a$ is equal to
  $s(n,d)$, that is, that $s(n,c) + (n+2) a \cdot s(n-1,2) = s(n,d)$,
  which we verify by the calculation

\begin{align*}
  s(n,d) &= \left \lfloor \frac{(n+2)n(c+a(n+1) - 1)}{2 (n+1)} \right \rfloor = 
           \left \lfloor \frac{(n+2)n(a(n+1))}{2 (n+1)}  +  
           \frac{(n+2)n(c- 1)}{2 (n+1)} \right \rfloor \\&=
  \frac{(n+2)n a }{2}  +  \left \lfloor \frac{(n+2)n(c- 1)}{2 (n+1)} \right \rfloor = 
  (n+2) a \cdot s(n-1,2) + s(n,c).
\end{align*}
\end{proof}

\section{An upper bound for the smallest inflection point of the
  Hilbert function of a complete intersection}\label{sec:inflection}

In order to use the results from the previous section to draw
conclusions about the WLP, we need an upper bound for the degree of
the expected Hilbert series
\begin{equation}\label{eq:SecondDiff}
  \left[\frac{(1-t^d)^{n+2}}{(1-t)^n}\right]
\end{equation}
for $R_{n,n+2,d}$ given by Fr\"oberg's conjecture. Since
$$\frac{(1-t^d)^{n+2}}{(1-t)^{n}} 
= (1-t)^2(1+t+\cdots+t^{d-1})^{n+2} ,$$ we are interested in the
lowest degree where the coefficients of the polynomial
$(1-t)^2(1+t+\cdots+t^{d-1})^{n+2}$ are non-positive.  We will provide
the necessary bounds by induction on $n$. The induction step is
Lemma~\ref{lemma:1} below and the base of the induction is given in
Lemma~\ref{lemma:2}.

For the statements of these lemmas, we introduce the following
notation. For a sequence $a_0,a_1,\ldots,a_n$ of integers, let
$\Delta(a)$ be the sequence of differences
$\Delta(a) = a_0,a_1-a_0,\dots, a_n-a_{n-1}, -a_n$. For simplicity we
will assume that all sequences are zero outside the range of indices
for which they are defined.  The generating series of this sequence is
$\sum_{i=0}^{n+1} \Delta (a)_it^i = (1-t)\sum_{i=0}^n a_it^i$. Instead
of looking at the range of indices for which the coefficients of the
polynomial in (\ref{eq:SecondDiff}) are non-positive, we will look at
where the first difference of the coefficients of
$(1+t+\cdots+t^{d-1})^{n+2}$ are decreasing.

\begin{lm} \label{lemma:1} Let $n \geq 4$, let $d \ge 1$, let
  \[ \sum_{i=0}^{n(d-1)} a_i t^i = (1+t+\cdots+t^{d-1})^n,\quad \text{
      and let }\quad \sum_{i=0}^{(n+1)(d-1)} b_i t^i =
    (1+t+\cdots+t^{d-1})^{n+1}.\] Suppose that for some
  $s\in \frac12\mathbb N$ with $n(d-1)/2-s\ge (d-1)/2$ we have that
  \begin{equation} \label{eq:one} \Delta(a)_{i} \geq \Delta(a)_{i+1},
    \qquad s \le i \le (d-1)n-s
  \end{equation}
  and
  \begin{equation}
    \label{eq:two}
    \Delta(a)_{s-j} \geq \Delta(a)_{s+j+1}, \qquad 0 \leq j \leq 
    s.
  \end{equation} 
  Then
  \begin{equation} \label{eq:three} \Delta(b)_{i} \geq \Delta(b)_{i+1}
    , \qquad s+\frac{d-1}{2} \le i \le
    (d-1)(n+1)-\left(s+\frac{d-1}{2}\right)
  \end{equation}
  and
  \begin{equation}
    \label{eq:four} \Delta(b)_{s+(d-1)/2-j} \geq \Delta(b)_{s+(d-1)/2+j+1} ,\qquad 0 \leq j \leq s+\frac{d-1}2.
  \end{equation}
  In all cases, the index $j\in \frac 12 \mathbb N$ takes only values
  that make the indices integers.
\end{lm}

\begin{proof}
  For simplicity, we will denote $(d-1)/2$ by $m$ which is an integer
  when $d$ is odd and a half integer for even $d$.  Observe that the
  sequences $\Delta({a})_i$ and $\Delta({b})_i$ are anti-symmetric around
  $mn+1/2$ and $m(n+1)+1/2$ respectively and that they are positive in
  the first half and negative in the second half. In particular, this gives that it is
  sufficient to prove (\ref{eq:three}) for $i<m(n+1).$

  We have that
  $(1-t)(1+t+\cdots+ t^{d-1})^{n+1} = (1-t^d)(1+t+\cdots+
  t^{d-1})^{n}$ which shows that $\Delta(b)_i = a_i-a_{i-d}$ and
  \begin{equation}\label{eq:compare}
    \Delta(b)_i-\Delta(b)_{i+1} = a_i-a_{i-d} - a_{i+1} + a_{i+1-d} =
    \Delta(a)_{i+1-d}-\Delta(a)_{i+1} =   \Delta(a)_{i-2m}-\Delta(a)_{i+1}.
  \end{equation}
  Hence we can use (\ref{eq:one}) to prove (\ref{eq:three}) when
  $i- 2m \ge s$ and $ i \le 2mn-s$, i.e., in the range
  $s+ 2m \le i \le 2mn-s$.

  Since we by the anti-symmetry of $\Delta(b)_i$ do not need to look at
  $i\ge m(n+1)$ and since we have that $2mn-s \ge m(n+1)$ by the
  assumption that $mn-s\ge m$, it only remains to prove
  (\ref{eq:three}) for $i$ in the range $s+m\le i < s+2m$. In order to
  do this, we use (\ref{eq:two}) and for $0 \le j\le m$, we have that
  \[
    \Delta(a)_{s-m+j}\ge \Delta(a)_{s+m-j+1}. 
  \]
  Moreover, by (\ref{eq:one}), we have 
 \[
 \Delta(a)_{s+m-j+1} \ge \Delta (a)_{s+m+j+1},
 \]
  which with $i=s+m+j$ in (\ref{eq:compare}) now gives
  \[
    \Delta(b)_{s+m+j} \ge \Delta(b)_{s+m+j+1}, \qquad 0 \le j \le m
  \]
  where we only consider the $j$ that makes $j+m$ an integer. This
  finishes the proof of (\ref{eq:three}).

  For (\ref{eq:four}) we have two cases, $j\ge m$ and $j\le m$. In the
  first case we write
  \[
    \Delta(b)_{s+m-j}\ge \Delta(b)_{s+m+j+1} \quad \Longleftrightarrow
    \quad a_{s+m-j}-a_{s-m-j-1}\ge a_{s+m+j+1}-a_{s-m+j}
  \]
  and the latter can be written
  \[
    \Delta(a)_{s+m-j} + \cdots + \Delta(a)_{s-m-j} \ge
    \Delta(a)_{s+m+j+1}+ \cdots + \Delta(a)_{s-m+j+1}.
  \]
  This holds term-wise because of (\ref{eq:two}) if $j\ge m$.

  For the second case, we write
  \[
    \Delta(b)_{s+m-j}\ge \Delta(b)_{s+m+j+1} \quad \Longleftrightarrow
    \quad a_{s-m+j}-a_{s-m-j-1}\ge a_{s+m+j+1}-a_{s+m-j}
  \]
  and the latter can be written as
  \[
    \Delta(a)_{s-m-j} + \cdots + \Delta(a)_{s-m+j} \ge
    \Delta(a)_{s+m+j+1}+ \cdots + \Delta(a)_{s+m-j+1}
  \]
  and again this holds term-wise because of (\ref{eq:two}) when
  $j\le m$. This finishes the proof of (\ref{eq:four}).
\end{proof}

Even though we do have the WLP for $R_{3,4,d}$, we start the induction
with this as the base case. The following lemma gives us the value
of $s$ that can be used in the induction step of Lemma~\ref{lemma:1}.

\begin{lm}\label{lemma:2}
  For $d>1$ let
  $a_0+a_1t + \cdots + a_{4(d-1)}t^{4(d-1)} =
  (1+t+\cdots+t^{d-1})^4$. Then for
  $s = \left\lfloor \frac{4(d-1)}{3}\right\rfloor$ we have that
  \[
    \Delta(a)_{j}\ge \Delta(a)_{j+1}, \qquad s\le j\le 4(d-1)-s
  \]
  and \[\Delta(a)_{s-j}\ge \Delta(a)_{s+j+1},\qquad 0 \le j\le s.\]
\end{lm}

\begin{proof}
  The coefficients of the polynomial
  $(1-t^d)^3/(1-t)^3 = (1+t+\cdots+t^{d-1})^3$ are unimodal and
  symmetric around degree $3(d-1)/2$. Hence the coefficients of the
  polynomial
  $$(1-t^d)^4/(1-t)^3 =  
  \Delta(a)_0 + \Delta(a)_1 t + \cdots + \Delta(a)_{4d-3}$$ are
  anti-symmetric around $(4(d-2)+1)/2= 2d-2$; having positive
  coefficients up to degree $2d-2$ and thereafter negative
  coefficients satisfying $\Delta(a)_{2d-2-i} = -\Delta(a)_{2d-1 + i}$
  for $0 \leq i \leq 2d-2$.

  We can write down an explicit formula for the positive coefficients
  as
  \[
    \Delta(a)_j = \begin{cases} \binom{j+2}2,& 0\le j\le d-1,\\
      \binom{j+2}{2}-4\binom{j+2-d}{2},& d\le j\le 2d-2.\\
    \end{cases}
  \]
  The second expression can be written as a quadratic polynomial in
  $j$ as
  \[
    g(j) = \binom{j+2}{2}-4\binom{j+2-d}{2} = -\frac32\left(
      j^2-\left(\frac{8d}3-3\right)j + \frac{4d^2}3 - 4d+2 \right)
  \]
  which is symmetric around $j = \frac{4d}3-\frac32$. Thus we have
  that $\Delta(a)_{j}\geq \Delta(a)_{j+1}$ when
  \[
    \frac{j + (j+1)}{2} \ge \frac{4d}{3}-\frac{3}{2}
    \Longleftrightarrow j \geq \frac{4d}3-2 \Longleftrightarrow j \geq
    s
  \]
  for the positive coefficients. For the negative coefficients, we use
  the anti-symmetry to conclude that
  $\Delta(a)_{j}\geq\Delta(a)_{j+1}$ when
  $j \leq s + 2 \cdot (2d-2 - s)=4(d-1)+s.$
    
  Moreover, for $ 0\le j\le s$, we have $\Delta(a)_{s-j}\ge g(s-j)$
  and $\Delta(a)_{s+j+1}= g(s+j+1)$ when $\Delta(a)_{s+j+1}$ is
  positive. Thus it is enough to verify that $g(s-j) \geq
  g(s+j+1)$. Indeed
  \[
    \frac{(s-j)+(s+j+1)}2 \ge \frac{4d}3-\frac32 \Longleftrightarrow s
    \ge \frac{4d}3-2 = \left\lfloor \frac{4(d-1)}{3}\right\rfloor.
  \]
\end{proof}

We now use Lemmas \ref{lemma:1} and \ref{lemma:2} to give an upper
bound on the degree of the expected Hilbert series for $n+2$ general
forms of degree $d$.

\begin{prop} \label{prop:2} An upper bound for the smallest inflection
  point of the Hilbert function of a complete intersection generated
  by $n+2$ forms of degree $d$ in $n+2 \geq 4$ variables is given by
  $\left\lfloor \frac{4(d-1)}{3}\right\rfloor + (n-2)(d-1)/2$, that is
  for $n \geq 2$ we have
  \[\deg \left (\left[\frac{(1-t^d)^{n+2}}{(1-t)^n}\right] \right)
    \leq \left\lfloor \frac{4(d-1)}{3}\right\rfloor + \frac{(n-2)(d-1)}{2}.\]
\end{prop}

\begin{proof}
  Induction on $n$ with Lemma \ref{lemma:2} as the base case and with
  Lemma \ref{lemma:1} as the induction step proves the statement about
  the inflection point. As seen before, this inflection point
  corresponds to the degree of the expected Hilbert series since
  \[
    \left[\frac{(1-t^d)^{n+2}}{(1-t)^n}\right] =
    \left[(1-t)^2\frac{(1-t^d)^{n+2}}{(1-t)^{n+2}}\right].
  \]
\end{proof}

The upper bound in Proposition \ref{prop:2} is far from being sharp,
and in the proof of Theorem \ref{thm:almost} we will need a better
bound in some cases. For this purpose we will in Lemma \ref{lemma:3}
below give a more general version of Proposition \ref{prop:2} that
could be used with a different base case than Lemma
\ref{lemma:2}. Lemma \ref{lemma:3} also gives the connection to the
WLP that will be used in Theorem \ref{thm:almost}.

\begin{lm} \label{lemma:3}
	
  Let $n \geq 4$, $d \geq 1$, and suppose that $\tilde{s}$ is an
  integer such that the assumptions in Lemma \ref{lemma:1} are
  satisfied. If $\tilde{s} < s(n-2,d)$, then
$$\deg \left (\left[\frac{(1-t^d)^{n+k}}{(1-t)^{n-2+k}}\right] \right)  < s(n-2+k,d) \text{\quad 
  for all even integers $k \geq 0$},$$ and in particular,
$R_{n-1+k,n+k,d}$ fails the WLP for all even integers $k \geq 0$.

\end{lm}

\begin{proof}

  If $n$ is even, a computation reveals that
  $s(n+k,d) - s(n-2+k,d) \geq d-1$, and if $n$ is odd, then
  $s(n+k,d) - s(n-2+k,d) = d-1$. From this we
  conclude that $s(n-2,d) + \frac{k+2}{2}(d-1) \leq s(n+k,d)$.

  On the other hand, repeated use of Lemma \ref{lemma:1} gives that
  $\deg \left (\left[\frac{(1-t^d)^{n+k}}{(1-t)^{n-2+k}}\right]
  \right) \leq \tilde{s} + k \cdot \frac{d-1}{2}$. Thus we have

$$\deg \left (\left[\frac{(1-t^d)^{n+k}}{(1-t)^{n-2+k}}\right] \right)  \leq \tilde{s} + k \cdot \frac{d-1}{2} <  s(n-2,d) + ((k-2)+2) \cdot \frac{d-1}{2} \leq s(n+k-2,d).$$

The second part of the proposition follows since by (\ref{eq:WLPred}),
the WLP for $R_{m+1,m+2,d}$ fails if the Hilbert series for
$R_{m,m+2,d}$ does not equal $[(1-t^d)^{m+2}/(1-t)^m]$.

\end{proof}

\begin{tm} \label{thm:almost} Let $n \geq 4, d \geq 2$. Then
  $R_{n,n+1,d}$ fails the WLP except possibly for

\[
  (n,d)\in \{(4, 2),(5, 2),(5, 3),(5, 5),(7, 2),(7, 3),(9, 2),(9,
  3),(11, 2),(11, 3)\}.
\]

\end{tm}

\begin{proof}
  According to Proposition \ref{prop:2}, we have
$$\deg \left (\left[\frac{(1-t^d)^{n+2}}{(1-t)^n}\right] \right) \leq 
\left\lfloor \frac{4(d-1)}{3}\right\rfloor + (n-2)(d-1)/2,$$ and by
Lemma \ref{lemma:3}, the WLP fails for $R_{n+1,n+2,d}$ if
$$s(n,d) > \left\lfloor \frac{4(d-1)}{3}\right\rfloor + (n-2)\frac{d-1}{2}.$$
Thus for even $n$ the WLP fails for $R_{n+1,n+2,d}$ if

\[
  \left\lfloor \frac{n(n+2)(d-1)}{2(n+1)}\right\rfloor -\left\lfloor
    \frac{4(d-1)}{3}\right\rfloor \ge (n-2)\frac{d-1}{2}+1.
\]
In order to show this, it is sufficient to show that
\[
  \frac{n(n+2)(d-1)}{2(n+1)}-\frac{4(d-1)}{3}\ge (n-2)\frac{d-1}{2}+2
\]
which can be written as
\[
  d \ge 2 + 12\cdot \frac{n+1}{n-2},
\]
which for $n=4$ gives $d\ge 32$.

In the same way, for odd $n>2$, we want to to show that
\[
  \frac{(n+1)(d-1)}{2} -\left\lfloor\frac{4(d-1)}{3}\right\rfloor \ge
  (n-2)\frac{d-1}{2}+1
\]
and here it is sufficient to show that
\[
  \frac{(n+1)(d-1)}{2} -\frac{4(d-1)}{3} \ge (n-2)\frac{d-1}{2}+1
\]
which is equivalent to $d \ge 7$.

Thus to this point we have by Lemma \ref{lemma:3} that $R_{n+1,n+2,d}$
fails for even $n \geq 4$ and $d \geq 32$, and for odd $n \geq 3$ and
$d \geq 7.$

For the remaining cases we introduce the notation
\[
  \tilde s(n,d) = \min \{ s \colon s \text{ satisfies the hypotheses
    in Lemma~\ref{lemma:1}}\}.
\]

For $n=3$, we check with Macaulay2 \cite{M2} that in the range
$2\le d <7$ except for $d=2$, we have $s(3,d) > \tilde{s}(3,d)$, so by
Lemma \ref{lemma:3}, the WLP for $R_{n+1,n+2,d}$ fails for all odd
$n \geq 3$ and $d>2$.  Since $s(5,2)=3>\tilde s(5,2)=2$, we also get
that the WLP for $R_{n+1,n+2,d}$ fails for all odd $n\ge 5$ and $d=2$.

For $n=4$, we check that $s(4,d) > \tilde s(4,d)$ in the range
$2\le d <32$ except for $d \in \{ 2,3,5\}$. Since
$s(6,5)=13>\tilde s(6,5)=12$, we have that the WLP for $R_{n+1,n+2,d}$
fails for all even $n\ge 6$ and $d>3$ according to Lemma
\ref{lemma:3}. For $d\in \{2,3\}$ and even $n$, we need to go to
$n=12$ to get $s(12,2)=6 > \tilde s(12,2)=5$ and
$s(12,3)=12>\tilde s(12,3)=11$.

Thus we have shown that the WLP for $R_{n,n+1,d}$ fails for all
$n\ge 4$ and $d\ge 2$ except possibly for
\[
  (n,d)\in \{(4, 2),(5, 2),(5, 3),(5, 5),(7, 2),(7, 3),(9, 2),(9,
  3),(11, 2),(11, 3)\}.
\]
\end{proof}

\section{Explicit formulas, the remaining cases, and the proof of
  Theorem \ref{thm:main}} \label{sec:sporadic} As we saw from the
previous section there there are ten cases to consider in order to
finish the proof of our main theorem. Four of them do satisfy the WLP,
and now we have to deal with the remaining cases that are
\[(n,d) \in \{(5,5),(7,3),(9,2),(9,3),(11,2),(11,3)\}.\] The case
$(5,5)$ was handled by Migiliore, Miró-Roig and Nagel in \cite{MMN},
the case $(7,3)$ by Ilardi and Vallès in \cite{edim7} and the cases
$(9,2)$ and $(11,2)$ by Sturmfels and Xu~\cite{squares}. Thus there
are two remaining cases: $(n,d) = (9,3)$ and $(n,d) = (11,3)$. We will
now deal with these two cases, but we will also give new arguments for
the other four since the method we use is the same.

We will for each case provide a set of elements of generators for the
inverse system in the top degree that shows that the Hilbert function
of $R_{n-1,n+1,d}$ is not the one expected from the Fr\"oberg
conjecture.

In order to do this we start by establishing an explicit formula for
the form of degree $(n+1)/2$ in $\Bbbk[X_1,X_2,\dots,X_n]$ that is
annihilated by the squares of $n+2$ general linear forms when $n$ is
odd. We will do this in two different ways with different sets of
parameters. The two versions are useful in different situations. In
the first version, we observe that for $n+2$ general forms we can by a
change of variables assume that $n$ are the variables and one is the
sum of the variables. The last form will have general coefficients.

Observe that the references to the result by Sturmfels and
Xu~\cite{squares} in Section~\ref{sec:degree} can be replaced by the
use of Theorem~\ref{thm:formula} to get a completely self-contained
proof of our main result.

We will in the following use $V(y_1,y_2,\dots,y_m)$ to denote the
Vandermonde determinant in variables $y_1,y_2,\dots,y_m$.

\begin{tm}\label{thm:formula}
  Let $ n = 2k-1$ for a positive integer $k$. The form
  \[\begin{split}
      F = \det \begin{bmatrix} X_1 & a_1 X_1& a_1^2 X_1& \cdots &
        a_1^{k-1} X_1& a_1 & a_1^2 & \cdots &
        a_1^{k-1}\\
        X_2 &a_2 X_2& a_2^2 X_2& \cdots & a_2^{k-1}X_2 & a_2 & a_2^2 &
        \cdots &
        a_2^{k-1}\\
        \vdots& \vdots & \vdots& \ddots& \vdots& \vdots&
        \vdots&     \ddots&       \vdots\\
        X_n & a_n X_n& a_n^2 X_n& \cdots & a_n^{k-1} X_n& a_n &a_n^2& \cdots & a_n^{k-1}\\
      \end{bmatrix} \\= \frac{1}{k!(k-1)!}\sum_{\sigma\in \mathfrak
        S_n} \operatorname{sgn}(\sigma) V(a_{\sigma_1},
      a_{\sigma_{2}},\dots, a_{\sigma_{k}}) V(a_{\sigma_{k+1}},
      a_{\sigma_{k+2}},\dots, a_{\sigma_{n}}) \prod_{j=k+1}^n
      a_{\sigma_j} \prod_{j=1}^k X_{\sigma_j}
    \end{split}
  \]
  is the unique form of degree $k$ in $\Bbbk[X_1,X_2,\dots,X_n]$ that
  is annihilated by the squares of the linear forms
  $x_1,x_2,\dots,x_n,x_1+x_2+\dots +x_n,a_1x_1+a_2x_2+\dots+a_nx_n$.
\end{tm}
\begin{proof}
  By Nagel and Trok~\cite{NT} there is a unique such form and it is
  sufficient for us to prove that this particular form is annihilated
  by the squares of the linear forms.  The equality between the two
  formulas follows from the generalized Laplace expansion over the
  first $k$ columns. Since the form is square-free, it is annihilated
  by the squares of the variables and it remains for us to check that
  it is annihilated by the squares of the last two linear forms
  $\ell_{n+1} = x_1+x_2+\dots +x_n$ and
  $\ell_{n+2} = a_1x_1+a_2x_2+\dots+a_nx_n$.

  We can start by computing
  $\ell_{n+1}^2\circ F = (x_1+x_2+\dots +x_n)^2\circ F$ where we by
  the Leibniz rule for determinants get a sum over terms where we
  substitute $X_i=1$, for $i=1,2,\dots, n$ in two of the first $k$
  columns. In all these terms, there will be a repeated column so all
  terms are zero.

  In the same way we get that
  $\ell_{n+1}^2\circ F = (a_1x_1+a_2x_2+\dots +a_nx_n)^2\circ F =
  0$. This time we substitute $X_i = a_i$, for $i=1,2,\dots,n$ in two
  of the first $k$ columns, which again results in repeated columns.
\end{proof}

For the second version of this formula, we observe that $n+2$ general
points in $\mathbb P^{n-1}$ are on a rational normal curve and we can
by a change of coordinates assume that this curve is the moment curve
with parametrization $(1:t:t^2:\cdots:t^{n-1})$. Thus we can assume
that the $n+2$ linear forms are given by
$\ell_i = \sum_{j=1}^n \alpha_i^{j-1}x_j$, for $i=1,2,\dots,n+2$,
where $\alpha_1,\alpha_2,\dots,\alpha_{n+2}$ are general elements of
the field $\Bbbk$.

     \begin{tm}
       For $n = 2k-1$, the form of degree $k = (n+1)/2$ that is
       annilated by the squares of the linear forms
       $\ell_i = \sum_{j=1}^{n} \alpha_i^{j-1} x_{j}$, for
       $i = 1,2,\dots,n+2$ is given by
       \[
         F = \det \begin{bmatrix} x_1&x_2&x_3&\cdots&x_k\\
           x_2&x_3&x_4&\cdots&x_{k+1}\\
           x_3&x_4&x_5&\cdots&x_{k+2}\\
           \vdots&\vdots&\vdots&\ddots&\vdots\\
           x_k&x_{k+1}&x_{k+2}&\cdots&x_{n}\\
         \end{bmatrix}.
       \]
     \end{tm}

     \begin{proof}
       Again we use that Nagel and Trok~\cite{NT} have shown that
       there is a unique such form and it is sufficient for us to
       prove that this determinant is annihilated by the squares of
       the linear forms.

       We observe that $\ell_i^2\circ F$ is given by a sum with signs
       over all ways of substituting $x_j$ with $\alpha_i^{j-1}$ in
       two of the rows of the matrix. These two rows become linearly
       dependent and thus $\ell_i^{2}\circ F = 0$ for
       $i = 1,2,\dots,n+2$.
     \end{proof}

     The advantage of this second version is that the formula does not
     depend on the parameters. Moreover, we can also use it to find
     formulas for the unique forms that are annihilated by some of the
     linear forms and the squares of the remaining linear forms.
     
     \begin{tm}\label{thm:formula3}
       For $0<k\le (n+1)/2$, the unique form of degree $k$ that is
       annilated by the linear forms
       $\ell_i = \sum_{j=1}^{n} \alpha_i^{j-1} x_{j}$, for
       $i = 1,2,\dots,n-2k+1$ and by the squares of the remaining
       $2k+1$ linear forms is given by
       \[
         F = \det \begin{bmatrix}
           1&\alpha_1&\alpha_1^2&\cdots&\alpha_1^{n-k}\\
           1&\alpha_2&\alpha_2^2&\cdots&\alpha_2^{n-k}\\
           \vdots&\vdots&\vdots&\ddots&\vdots\\
           1&\alpha_{n-2k+1}&\alpha_{n-2k+1}^2&\cdots&\alpha_{n-2k+1}^{n-k}\\
           x_1&x_2&x_3&\cdots&x_{n+1-k}\\
           x_2&x_3&x_4&\cdots&x_{n+2-k}\\
           \vdots&\vdots&\vdots&\ddots&\vdots\\
           x_k&x_{k+1}&x_{k+2}&\cdots&x_{n}\\
         \end{bmatrix}.
       \]
     \end{tm}

   \begin{proof}
     The uniqueness is given by Nagel and Trok~\cite{NT} and it is
     enough for us to show the vanishing.
     
     We have that $\ell^2\circ F=0$ for any linear form
     $\ell = \sum_{j=1}^n \alpha^{j-1}x_j$ for the same reason as in
     the previous theorem. Applying $\ell_i$, where
     $i=1,2,\dots,n-2k+1$, to $F$ gives a sum over the $k$
     determinants we get by replacing $x_j$ with $\alpha_i^{j-1}$ in
     each of the $k$ lowest rows. Hence $\ell_i\circ F = 0$, for
     $i =1,2,\dots,n-2k+1$.
   \end{proof}

   We can now treat the sporadic cases not covered by Theorem
   \ref{thm:almost} where the WLP fails.
   \begin{tm} \label{thm:sporadic} The WLP fails for
     $R_{n,n+1,d}= \Bbbk[x_1,x_2,\dots,x_n]/\langle
     \ell_1^d,\ell_2^d,\dots,\ell_{n+1}^d\rangle$ in the cases
     $(n,d) \in \{ (5,5),(7,3),(9,2),(9,3),(11,2),(11,3)\}$. In
     particular we have that the Hilbert series of $R_{4,6,5}$,
     $R_{6,8,3}$, $R_{8,10,2}$, $R_{8,10,3}$, $R_{10,12,2}$, and
     $R_{10,12,3}$ are
     \[\begin{array}{c}
         1 + 4 t + 10 t^2 + 20 t^3 + 35 t^4 + 50 t^5 +
         60 t^6 + 60 t^7 + 45 t^8 + 14 t^9,\\
         1 + 6 t + 21 t^2 + 48 t^3 + 78 t^4 + 84 t^5 +
         43 t^6,\\
         1 + 8 t + 26 t^2 + 40 t^3 + 16 t^4,\\
         1  + 8 t  + 36 t^2 + 110 t^3 + 250 t^4 + 432 t^5 + 561 t^6 + 492 t^7
         + 171 t^8 ,
         \\
         1 + 10 t + 43 t^2 + 100 t^3 + 121 t^4 + 32 t^5,\\
         \text{and}\\
         1  + 10 t  + 55 t^2 + 208 t^3 + 595 t^4 + 1342 t^5 + 2431 t^6 + 3520 t^7
         + 3916 t^8 + 2860 t^9 +  683 t^{10}\\
       \end{array}
     \]
     which differ in the leading term from $10t^9$, $42t^6$, $15t^4$,
     $135t^8$, $22t^5$ and $88t^{10}$ that are expected by the
     Fr\"oberg conjecture.
   \end{tm}

\begin{proof}
  We consider the ring $R_{n,n+2,d}$ and denote our set of $n+2$
  general linear forms by
  \[\mathcal L = \{\ell_1,\ell_2,\dots,\ell_{n+2}\} =
    \left\{\sum_{i=1}^n a_1^{i-1}x_i, \sum_{i=1}^n a_2^{i-1}x_i,\dots,
      \sum_{i=1}^n a_{n+2}^{i-1} x_i\right\},\] where
  $a_1,a_2,\dots,a_{n+2}$ are general elements of $\Bbbk$.
  
  Using Theorem~\ref{thm:formula3}, we can find a formula for the
  unique form of degree $k$ that is annihilated by $n-2k+1$ of the
  $n+2$ linear forms in $\mathcal L$ and by the squares of the
  remaining $2k+1$ linear forms in $\mathcal L$. For
  $\mathcal S\subseteq \mathcal L$ of size $n-2k+1$, we denote this
  unique form by $F_{\mathcal S}$. Observe that the coefficients of
  $F_{\mathcal S}$ are polynomials in the parameters
  $a_1,a_2,\dots,a_{n+2}$. In each of the six cases of the theorem, we
  produce a set of forms in the inverse system of $R_{n,n+2,d}$ in the
  top degree such that the dimension of the subspace they span agrees
  with the stated value of the Hilbert function in the socle
  degree. It will be enough to verify this for a specialization of the
  parameters, since a specialization can only lower the
  dimension. Thus this gives a lower bound for the Hilbert function in
  the socle degree. On the other hand, the computation of the Hilbert
  function of $R_{n,n+2,d}$ for a specialization provides an upper
  bound. Since they agree we can make the desired conclusion.

  In all the cases, we use the specialization $a_i=i-1$,
  $i=1,2,\dots,n+2$ to verify the dimension using Macaulay2.

  For $R_{4,6,5}$ we need $14$ linearly independent forms of degree
  $9$. We use forms that can be written as
  $F = F_{\mathcal S_1}F_{\mathcal S_2}F_{\mathcal S_3}F_{\mathcal
    S_4}F_{\mathcal S_5}F_{\mathcal S_6}$ where
  $|\mathcal S_1]=|\mathcal S_2]=|\mathcal S_3] =3$ and
  $|\mathcal S_4]=|\mathcal S_5]=|\mathcal S_6]=1$ such that each
  linear form in $\mathcal L$ is contained in exactly two of the six
  subsets. Observe that the three first factors are linear and the
  three last are quadratic. For each $\ell \in \mathcal L$ we get
  $\ell^5\circ F = 0$ by the pigeonhole principle and the general
  Leibniz rule since $\ell$ annihilates two of the factors and
  $\ell^2$ annihilates the remaining four factors.

  For $R_{6,8,3}$ we need $43$ linearly independent forms of degree
  $6$. We use forms that can be written as
  $F_{\mathcal S_1}F_{\mathcal S_2}F_{\mathcal S_3}F_{\mathcal S_4}$
  where $|\mathcal S_1|=|\mathcal S_2| =5$ and
  $|\mathcal S_3|=|\mathcal S_4| = 3$ each linear form in $\mathcal L$
  is contained in exactly two of the subsets. These forms are
  annihilated by the squares of all linear forms in $\mathcal L$ since
  each linear form annihilates two of the factors and the square of
  the linear form annihilates the remaining two factors.

  For $R_{8,10,2}$ we need $16$ linearly independent forms of degree
  $4$. These can be obtained as
  $F = F_{\mathcal S}F_{\mathcal L\setminus \mathcal S}$ for subsets
  $\mathcal S$ of size five. These are products of two quadrics and
  they are annihilated by the squares of the linear forms in
  $\mathcal L$ since for each $\ell$ in $\mathcal L$ we have that
  $\ell$ annihilates one of the factors and $\ell^2$ annihilates the
  other.

  For $R_{8,10,3}$ we will produce a set of $171$ linearly independent
  forms of degree $8$ that are annihilated by the cubes of the linear
  forms. These forms are obtained as
  $F = F_{\mathcal S_1}F_{\mathcal S_2}F_{\mathcal S_2}F_{\mathcal
    S_2}$ where
  $|\mathcal S_1|=|\mathcal S_2|=|\mathcal S_3|=|\mathcal S_4|=5$ and
  each linear form in $\mathcal L$ is contained in two of the
  subsets. Thus we get that
  $\ell^3\circ (F_{\mathcal S_1}F_{\mathcal S_2}F_{\mathcal
    S_2}F_{\mathcal S_2})=0$ for all $\ell\in \mathcal L$ since $\ell$
  anihilates two of the factors and $\ell^2$ annihilates the remaining
  two.
  
  For $R_{10,12,2}$ we provide a set of $32$ linearly independent
  forms of degree $5$ that are annihilated by the squares of the
  linear forms in $\mathcal L$.  We do this by forms
  $F = F_{\mathcal S}F_{\mathcal L\setminus \mathcal S}$ for subsets
  $\mathcal S$ of size five. Each linear form $\ell$ in $\mathcal L$
  annihilates one of the factors and $\ell^2$ the other. Hence
  $\ell^2\circ F_{\mathcal S}F_{\mathcal L\setminus \mathcal S} =0$.

  For $R_{10,12,3}$ we provide a set of $683$ linearly independent
  forms of degree $10$ that are annihilated by the cubes of the linear
  forms in $\mathcal L$. We do this by forms
  $F = F_{\mathcal S_1}F_{\mathcal S_2}F_{\mathcal S_3}F_{\mathcal
    S_4}$ where $|\mathcal S_1|=|\mathcal S_2|=5$,
  $|\mathcal S_3|=|\mathcal S_4|=7$ and every $\ell\in \mathcal L$ is
  contained in two of the subsets. Now
  $\ell^3\circ(F_{\mathcal S_1}F_{\mathcal S_2}F_{\mathcal
    S_3}F_{\mathcal S_4})=0$ for all $\ell \in \mathcal L$ since
  $\ell$ annihilates two of the factors and $\ell^2$ annihilates the
  remaining two.
\end{proof}

We can now prove Theorem \ref{thm:main}.

\begin{proof}[Proof of Theorem \ref{thm:main}]
  By Theorem \ref{thm:almost} we have that the WLP for $R_{n,n+1,d}$
  fails when $n \geq 4, d \geq 2$ expect possibly for the cases
$$(n,d)\in \{(4, 2),(5, 2),(5, 3),(5, 5),(7, 2),(7, 3),(9, 2),(9,3),(11, 2),(11, 3)\}.$$

In the case $(4,2)$, $(5,2)$, $(5,3)$ and $(7,2)$, we can verify by
one example that they do satisfy the WLP and for the remaining cases
Theorem \ref{thm:sporadic} shows that they fail to satisfy the WLP.

Finally, we refer to the introduction for references to the cases
$n=2$ and $n=3$, and for the cases $n=1$, $d=1$, the WLP is trivially
satisfied.
\end{proof}

\section*{acknowledgements}

  We wish to thank CIRM in Luminy for hosting us during the workshop
  on the Lefschetz properties in October 2019, and to Uwe Nagel for
  giving an inspiring talk on the subject of this paper. We also thank
  the anonymous referee for useful comments that helped us improve the
  presentation of the paper, and in particular, for providing us with
  the proof of Theorem \ref{thm:formula} which significantly
  simplified our original argument.  Finally, computer experiments in
  Macaulay2 were absolutely crucial for our understanding of the
  algebras that we have considered.


\begin{thebibliography}{PTW02}

\bibitem[AH95]{ah} James Alexander, André Hirschowitz, {\em Polynomial
    interpolation in several variables}, J. Alg. Geom. 4 no. 2 (1995)
  201--222.

\bibitem[Cha02]{chandler} Karen A. Chandler, {\em Linear Systems of
    Cubics Singular at General Points of Projective Space}, Compositio
  Math. 134 no.  3 (2002), 269--282.

\bibitem[Cil01]{sghh} Ciro Ciliberto, {\em Geometric aspects of
    polynomial interpolation in more variables and of Waring's
    problem}, European Congress of Mathematics, vol. I, Progr. Math.,
  vol. 201, Barcelona, 2000, Birkhäuser, Basel (2001), 289--316.

\bibitem[Eis05]{E} David Eisenbud, {\em The geometry of syzygies},
  Graduate Texts in Mathematics 229, Springer, New York, 2005.

\bibitem[EI95]{emsalem} Jacques Emsalem and Anthony Iarrobino, {\em
    Inverse system of a symbolic power I}, J. Algebra, 174 (1995),
  1080--1090.

\bibitem[Fr\"o85]{F} Ralf Fr\"oberg, {\em An inequality for Hilbert
    series of graded algebras}, Math. Scand. 56 (1985), 117--144.

\bibitem[GS]{M2} Dan Grayson and Mike Stillman, {\em Macaulay2, a
    software system for research in algebraic geometry}, {\it
    available at} {\tt www.math.uiuc.edu/Macaulay2}.


\bibitem[HSS11]{harbourne} Brian Harbourne, Hal Schenck and Alexandra
  Seceleanu, {\em Inverse systems, Gelfand-Tsetlin patterns and the
    Weak Lefschetz Property}, J. London Math. Soc. 84, no. 3 (2011),
  712--730.

\bibitem[HMMNW13]{book} Tadahito Harima, Toshiaki Maeno, Hideaki
  Morita, Yasuhide Numata, Akihito Wachi, and Junzo Watanabe, {\em The
    Lefschetz properties}, Lecture Notes in Mathematics, 2080 (2013).

\bibitem[HMNW03]{HMNW} Tadahito Harima, Juan Migliore, Uwe Nagel, and
  Junzo Watanabe, {\em The weak and strong Lefschetz properties for
    Artinian K-algebras}. J. Algebra, 262 no. 1 (2003), 99--126.

\bibitem[Iar97]{iarrobino} Anthony Iarrobino, {\em Inverse system of a
    symbolic power III: thin algebras and fat points}, Compositio
  Math. 108 no. 3 (1997), 319--356.

\bibitem[IV19]{edim7} Giovanna Ilardi and Jean Vall\`es, {\em Eight
    cubes of linear forms in $\bP^n$}, arXiv 1910.04035.

\bibitem[MMN12]{MMN} Juan C. Migliore, Rosa M. Mir\'o-Roig, and Uwe
  Nagel, {\em On the weak Lefschetz property for powers of linear
    forms}, Algebra Number Theory, 6, no. 3 (2012), 487--526.

\bibitem[MN13]{atour} Juan C. Migliore and Uwe Nagel, {\em A tour of
    the weak and strong Lefschetz properties}, J. Commut. Algebra 5,
  no. 3 (2013), 329--358.

\bibitem[Mir16]{M} Rosa M. Mir\'o-Roig, {\em Harbourne, Schenck and
    Seceleanu’s conjecture}, J. Algebra, 462 (2016), 54--66.

\bibitem[MH20]{MT} Rosa M. Mir\'o-Roig and Quang Hoa Tran, {\em On the
    weak Lefschetz property for almost complete intersections
    generated by uniform powers of general linear forms}, J. Algebra
  551 (2020), 209--231.

\bibitem[NT19]{NT} Uwe Nagel and Bill Trok, {\em Interpolation and the
    weak Lefschetz property}, Trans. Amer. Math. Soc., 372 no. 12
  (2019), 8849--8870.

\bibitem[Nen17]{N} Gleb Nenashev, {\em A note on Fr\"oberg’s
    conjecture for forms of equal degrees},
  C. R. Math. Acad. Sci. Paris 355 (2017), no. 3, 272--276.

\bibitem[SS10]{edim3} Hal Schenck and Alexandra Seceleanu, {\em The
    weak Lefschetz property and powers of linear forms in
    $K[x, y, z$]}, Proc. Amer. Math. Soc., 138 no. 7 (2010),
  2335--2339.

\bibitem[Sta80]{Stanley} Richard P. Stanley, {\em Weyl groups, the
    hard Lefschetz theorem, and the Sperner property}, SIAM
  J. Algebraic Discrete Methods, 1 no. 2 (1980), 168--184, 1980.

\bibitem[SX10]{squares} Bernd Sturmfels and Zhiqiang Xu, {\em Sagbi
    bases of Cox-Nagata rings}, J. Eur. Math. Soc., 12, no. 2 (2010),
  429--459.

\end{thebibliography}
\end{document}